\documentclass[11pt]{amsart}

\usepackage[T1]{fontenc}
\usepackage[utf8]{inputenc}
\usepackage{lmodern}
\usepackage{microtype}
\usepackage{mathtools}
\usepackage{amssymb}
\usepackage{enumitem}
\usepackage[hidelinks]{hyperref}

\hypersetup{
  pdftitle={A Counterexample to Teschner's Bondage-Number Conjecture},
  pdfauthor={Yousof Yavari},
  pdfsubject={Graph theory; domination number; bondage number},
  pdfkeywords={bondage number, domination number, Teschner's conjecture, cubic bipartite graph, counterexample}
}

\numberwithin{equation}{section}

\newtheorem{theorem}{Theorem}[section]
\newtheorem{lemma}[theorem]{Lemma}
\newtheorem{proposition}[theorem]{Proposition}

\newtheorem{conjecture}[theorem]{Conjecture}
\theoremstyle{definition}

\newcommand{\cD}{\mathcal{D}}

\title[A counterexample to Teschner's conjecture]{A Counterexample to Teschner's Bondage-Number Conjecture}

\author{Yousof Yavari}
\address{The University of British Columbia, Vancouver, British Columbia, Canada}
\email{yousofy@student.ubc.ca}

\date{August 4, 2026}

\begin{document}

\begin{abstract}
For a finite simple graph $G$ with at least one edge, the bondage number $b(G)$ is the
least number of edges whose deletion increases the domination number
$\gamma(G)$.  Teschner conjectured that
$b(G)\le \tfrac32\Delta(G)$ for every graph $G$.  We disprove this
conjecture by giving a connected cubic bipartite graph on eighteen
vertices with
\[
  \gamma(G)=6
  \qquad\text{and}\qquad
  b(G)=5.
\]
The domination number is established by a complete counting argument across
the bipartition.  An explicit five-edge deletion raises the domination
number from six to seven.  For the matching lower bound, we give an exact
finite certificate: the graph has 297 minimum dominating sets, and deleting
any one of its $\binom{27}{4}=17{,}550$ four-edge subsets leaves at least one
of those sets dominating.  The enumeration is deterministic, uses only exact
integer and set operations, and is reproduced by the complete
standard-library verifier included in the appendix.
\end{abstract}

\maketitle

\section{Introduction}

All graphs in this paper are finite, simple, and undirected.  A set
$D\subseteq V(G)$ is a \emph{dominating set} of a graph $G$ if every vertex
of $V(G)\setminus D$ has a neighbor in $D$.  The minimum cardinality of a
dominating set is the \emph{domination number} $\gamma(G)$.  Fink,
Jacobson, Kinch, and Roberts introduced the \emph{bondage number} $b(G)$ as
the minimum number of edges whose deletion increases the domination
number~\cite{FinkEtAl1990}.

Teschner proved the bound $b(G)\le \tfrac32\Delta(G)$ for graphs with
domination number at most three and proposed that the same inequality holds
without restriction~\cite{Teschner1995}. Subsequent work established the conjecture for broad classes of graphs without
settling it in full; see, for example, Gagarin and
Zverovich~\cite{GagarinZverovich2013}.

\begin{conjecture}[Teschner~\cite{Teschner1995}]\label{conj:Teschner}
Every finite simple graph $G$ with at least one edge satisfies
\begin{equation}\label{eq:Teschner}
  b(G)\le \frac32\Delta(G).
\end{equation}
\end{conjecture}

We give a counterexample in which every vertex has degree three.

\begin{theorem}\label{thm:main}
There exists a connected cubic bipartite graph $G$ such that
\[
  \gamma(G)=6
  \qquad\text{and}\qquad
  b(G)=5.
\]
Consequently,
\[
  b(G)=5>\frac92=\frac32\Delta(G),
\]
so Conjecture~\ref{conj:Teschner} is false.
\end{theorem}

The proof of $\gamma(G)=6$ and the upper bound $b(G)\le5$ is entirely
hand-checkable.  The lower bound $b(G)\ge5$ is an exact finite computation.
We isolate the elementary certificate criterion used by that computation,
state all enumerated totals, and include the complete verifier in
Appendix~\ref{app:verifier}.

\section{Preliminaries}

For a vertex $v$ of a graph $G$, let
\[
  N_G(v)=\{u\in V(G):uv\in E(G)\}
\]
be the set of neighbors of $v$, called its \emph{open neighborhood}, and let
$N_G[v]=N_G(v)\cup\{v\}$ be its \emph{closed neighborhood}.  For
$D\subseteq V(G)$, put
\[
  N_G[D]=\bigcup_{v\in D}N_G[v].
\]
Thus $D$ dominates $G$ if and only if $N_G[D]=V(G)$.  If
$F\subseteq E(G)$, then $G-F$ denotes the subgraph obtained from $G$ by
deleting the edges of $F$.  For a finite set $S$, the notation
$\binom{S}{k}$ denotes the set of its $k$-element subsets.

For every graph $G$ with $E(G)\ne\varnothing$, its bondage number is
\begin{equation}\label{eq:bondage-definition}
  b(G)=\min\bigl\{|F|:F\subseteq E(G),\ 
  \gamma(G-F)>\gamma(G)\bigr\}.
\end{equation}

We record two elementary facts about domination.

\begin{lemma}\label{lem:basic-domination}
Let $G$ be a finite simple graph.
\begin{enumerate}[label=\textup{(\roman*)}]
\item If $F\subseteq F'\subseteq E(G)$, then
\[
  \gamma(G-F)\le \gamma(G-F').
\]
\item If $G_1,\ldots,G_r$ are the components of $G$, then
\[
  \gamma(G)=\sum_{i=1}^r\gamma(G_i).
\]
\end{enumerate}
\end{lemma}

\begin{proof}
For \textup{(i)}, every dominating set of $G-F'$ also dominates $G-F$,
because $G-F$ contains every edge of $G-F'$.  For \textup{(ii)}, a set
dominates $G$ if and only if its intersection with each component $G_i$
dominates $G_i$.  Minimizing independently over the components gives the
stated sum.
\end{proof}

The next lemma characterizes exactly when a fixed dominating set remains
dominating after edges are deleted.

\begin{lemma}[Bundle criterion]\label{lem:bundle}
Let $D$ be a dominating set of a graph $G$.  For each $x\in V(G)\setminus D$,
define
\begin{equation}\label{eq:bundle}
  B_G(D,x)=\{xd\in E(G):d\in D\}.
\end{equation}
Then, for every $F\subseteq E(G)$, the set $D$ dominates $G-F$ if and only if
\begin{equation}\label{eq:bundle-criterion}
  B_G(D,x)\nsubseteq F
  \qquad\text{for every }x\in V(G)\setminus D.
\end{equation}
Equivalently, $D$ fails to dominate $G-F$ if and only if
$B_G(D,x)\subseteq F$ for some $x\notin D$.
\end{lemma}

\begin{proof}
Every vertex of $D$ remains dominated by itself after any edge deletion.
For $x\notin D$, the set $D$ dominates $x$ in $G-F$ precisely when at least
one edge joining $x$ to a vertex of $D$ remains.  Since $D$ dominates $G$,
the set $B_G(D,x)$ is nonempty, and such an edge remains precisely when
$B_G(D,x)$ is not contained in $F$.
\end{proof}

\section{The graph and its domination number}

Let
\[
  V(G)=\{0,1,\ldots,17\},
\]
and let $E(G)$ be the following set of 27 edges, where each pair denotes an
unordered edge:
\begin{equation}\label{eq:edge-set}
\begin{aligned}
E(G)=\{&(0,6),(0,10),(0,16),(1,8),(1,11),(1,12),\\
&(2,4),(2,5),(2,13),(3,4),(3,5),(3,9),\\
&(4,16),(5,11),(6,7),(6,13),(7,8),(7,14),\\
&(8,13),(9,10),(9,14),(10,15),(11,15),\\
&(12,15),(12,17),(14,17),(16,17)\}.
\end{aligned}
\end{equation}

\begin{proposition}\label{prop:structure}
The graph $G$ is connected, cubic, and bipartite, with bipartition
\begin{equation}\label{eq:bipartition}
\begin{aligned}
X&=\{0,1,4,5,7,9,13,15,17\},\\
Y&=\{2,3,6,8,10,11,12,14,16\}.
\end{aligned}
\end{equation}
In particular, $\Delta(G)=3$.
\end{proposition}

\begin{proof}
Every edge in \eqref{eq:edge-set} has one endpoint in $X$ and one endpoint
in $Y$.  The neighborhoods of the vertices in $X$ are displayed in
Table~\ref{tab:neighborhoods}.  Every row has three entries, and every
vertex of $Y$ occurs in exactly three of the displayed neighborhoods.
Thus every vertex of $G$ has degree three.

Starting from $0$, the distance layers are
\[
\begin{array}{ll}
L_0=\{0\},\quad L_1=\{6,10,16\}, &
L_2=\{4,7,9,13,15,17\},\\[2pt]
L_3=\{2,3,8,11,12,14\}, &
L_4=\{1,5\}.
\end{array}
\]
Their union is $V(G)$, so $G$ is connected.
\end{proof}

\begin{table}[ht]
\caption{Neighborhoods of the vertices in $X$.}
\label{tab:neighborhoods}
\centering
\small
\begin{tabular}{c|c@{\qquad}c|c@{\qquad}c|c}
$x$ & $N_G(x)$ & $x$ & $N_G(x)$ & $x$ & $N_G(x)$\\
\hline
$0$  & $\{6,10,16\}$ & $1$  & $\{8,11,12\}$  & $4$  & $\{2,3,16\}$\\
$5$  & $\{2,3,11\}$  & $7$  & $\{6,8,14\}$   & $9$  & $\{3,10,14\}$\\
$13$ & $\{2,6,8\}$   & $15$ & $\{10,11,12\}$ & $17$ & $\{12,14,16\}$
\end{tabular}
\end{table}

\begin{proposition}\label{prop:gamma}
The graph $G$ has domination number
\[
  \gamma(G)=6.
\]
\end{proposition}

\begin{proof}
It suffices first to rule out a dominating set of cardinality five, because
any smaller dominating set could be enlarged to one of cardinality five.
Suppose that $D$ is a dominating five-set, and put
\[
  a=|D\cap X|.
\]
The $9-a$ vertices of $X\setminus D$ can be dominated only by the
$5-a$ selected vertices of $Y$, each of degree three.  Similarly, the
$9-(5-a)=4+a$ vertices of $Y\setminus D$ can be dominated only by the $a$
selected vertices of $X$.  Consequently,
\begin{equation}\label{eq:bipartite-counts}
  9-a\le3(5-a),
  \qquad
  4+a\le3a.
\end{equation}
The first inequality gives $a\le3$, and the second gives $a\ge2$.  Thus
$a\in\{2,3\}$.  Let $Z$ be the part in \eqref{eq:bipartition} containing
exactly two vertices of $D$, and let $W$ be the other part.  Write
\[
  D\cap Z=\{z,z'\}.
\]
The set $W\setminus D$ has six vertices.  Since vertices within the same
part are nonadjacent, every vertex of $W\setminus D$ must be adjacent to
$z$ or $z'$.  Hence
\[
  W\setminus D\subseteq N_G(z)\cup N_G(z').
\]
Both $N_G(z)$ and $N_G(z')$ are three-element subsets of $W$, so
\[
  6=|W\setminus D|
  \le |N_G(z)\cup N_G(z')|
  \le |N_G(z)|+|N_G(z')|=6.
\]
The inclusion and equality of cardinalities imply
$W\setminus D=N_G(z)\cup N_G(z')$.  Moreover, equality in the final union
bound implies $N_G(z)\cap N_G(z')=\varnothing$.  Taking complements in
$W$ therefore gives
\begin{equation}\label{eq:forced-pair}
  N_G(z)\cap N_G(z')=\varnothing
  \qquad\text{and}\qquad
  D\cap W=W\setminus\bigl(N_G(z)\cup N_G(z')\bigr).
\end{equation}

Comparing the neighborhoods determined by \eqref{eq:edge-set} gives the
exhaustive list in Table~\ref{tab:disjoint-pairs}.  The left half treats
$Z=X$ and $W=Y$, while the right half treats $Z=Y$ and $W=X$.  In every
row, the final entry in the corresponding half is a vertex of
$Z\setminus\{z,z'\}$ with no neighbor in the forced set $D\cap W$.
That vertex is outside $D$ and is not dominated by $D$, a contradiction.
Hence no five-set dominates $G$.

Finally, the set
\begin{equation}\label{eq:dominating-six}
  D_0=\{0,1,2,3,14,15\}
\end{equation}
dominates $G$.  Indeed, the vertices outside $D_0$ are dominated as
witnessed below, where each relation $u\sim v$ means that the vertex $u$ is
adjacent to, and therefore dominated by, the selected vertex $v$:
\[
\begin{gathered}
4\sim2,\quad 5\sim2,\quad 6\sim0,\quad 7\sim14,
\quad 8\sim1,\quad 9\sim3,\\
10\sim0,\quad 11\sim1,\quad 12\sim1,\quad 13\sim2,
\quad 16\sim0,\quad 17\sim14.
\end{gathered}
\]
Therefore $\gamma(G)\le6$, while the preceding argument gives
$\gamma(G)\ge6$.
\end{proof}

\begin{table}[ht]
\caption{All pairs in each bipartition class with disjoint neighborhoods.}
\label{tab:disjoint-pairs}
\centering
\scriptsize
\setlength{\tabcolsep}{3pt}
\begin{tabular}{c|c|c@{\qquad}c|c|c}
\multicolumn{3}{c}{Two selected vertices in $X$} &
\multicolumn{3}{c}{Two selected vertices in $Y$}\\
$\{z,z'\}$ & Forced $D\cap Y$ & Missed in $X$ &
$\{z,z'\}$ & Forced $D\cap X$ & Missed in $Y$\\
\hline
$\{0,1\}$   & $\{2,3,14\}$    & $15$ & $\{2,10\}$  & $\{1,7,17\}$  & $3$\\
$\{0,5\}$   & $\{8,12,14\}$   & $4$  & $\{2,12\}$  & $\{0,7,9\}$   & $11$\\
$\{1,4\}$   & $\{6,10,14\}$   & $5$  & $\{2,14\}$  & $\{0,1,15\}$  & $3$\\
$\{1,9\}$   & $\{2,6,16\}$    & $15$ & $\{3,6\}$   & $\{1,15,17\}$ & $2$\\
$\{4,7\}$   & $\{10,11,12\}$ & $13$ & $\{3,8\}$   & $\{0,15,17\}$ & $2$\\
$\{4,15\}$  & $\{6,8,14\}$    & $5$  & $\{3,12\}$  & $\{0,7,13\}$  & $11$\\
$\{5,7\}$   & $\{10,12,16\}$ & $13$ & $\{6,11\}$  & $\{4,9,17\}$  & $8$\\
$\{5,17\}$  & $\{6,8,10\}$    & $4$  & $\{6,12\}$  & $\{4,5,9\}$   & $8$\\
$\{7,15\}$  & $\{2,3,16\}$    & $1$  & $\{8,10\}$  & $\{4,5,17\}$  & $6$\\
$\{9,13\}$  & $\{11,12,16\}$ & $7$  & $\{8,16\}$  & $\{5,9,15\}$  & $6$\\
$\{13,15\}$ & $\{3,14,16\}$   & $1$  & $\{11,14\}$ & $\{0,4,13\}$  & $12$\\
$\{13,17\}$ & $\{3,10,11\}$   & $7$  & $\{11,16\}$ & $\{7,9,13\}$  & $12$
\end{tabular}
\end{table}

\section{The bondage number}

We first give an explicit edge set that increases the domination number.

\begin{proposition}\label{prop:upper}
The graph $G$ satisfies $b(G)\le5$.
\end{proposition}

\begin{proof}
Delete the five-edge set
\begin{equation}\label{eq:F5}
  F_5=\{(0,6),(0,10),(0,16),(6,7),(6,13)\}.
\end{equation}
These are precisely the edges incident with at least one of the adjacent
vertices $0$ and $6$, so both vertices are isolated in $G-F_5$.  Let
\[
  H=G-\{0,6\}
\]
be the induced subgraph on the other sixteen vertices.  The set
\begin{equation}\label{eq:H-dominating-five}
  A=\{1,2,4,10,14\}
\end{equation}
dominates $H$: using the same notation, the vertices of
$V(H)\setminus A$ are dominated as witnessed by
\[
\begin{gathered}
3\sim4,\quad 5\sim2,\quad 7\sim14,\quad 8\sim1,
\quad 9\sim10,\quad 11\sim1,\\
12\sim1,\quad 13\sim2,\quad 15\sim10,
\quad 16\sim4,\quad 17\sim14.
\end{gathered}
\]
Thus $\gamma(H)\le5$.

If a four-set $A'$ dominated $H$, then $A'\cup\{0\}$ would dominate $G$:
the set $A'$ would dominate every vertex other than $0$ and $6$, while the
selected vertex $0$ would dominate itself and its neighbor $6$.  This would
contradict Proposition~\ref{prop:gamma}.  Hence $\gamma(H)=5$.

The graph $G-F_5$ is the disjoint union of $H$ and two isolated vertices.
Lemma~\ref{lem:basic-domination}\textup{(ii)} therefore gives
\[
  \gamma(G-F_5)=\gamma(H)+1+1=7>6=\gamma(G).
\]
It follows from \eqref{eq:bondage-definition} that $b(G)\le5$.
\end{proof}

We now certify that four deleted edges never suffice.  The proof is an exact
exhaustive enumeration, but the mathematical test used in each case is
precisely Lemma~\ref{lem:bundle}.

\begin{proposition}[Exact four-edge certificate]\label{prop:four-edge}
For every $F\in\binom{E(G)}{4}$,
\[
  \gamma(G-F)=6.
\]
\end{proposition}

\begin{proof}
Let
\begin{equation}\label{eq:minimum-set}
  \cD=\{D\in\tbinom{V(G)}6:N_G[D]=V(G)\}.
\end{equation}
By Proposition~\ref{prop:gamma}, the members of $\cD$ are exactly the
minimum dominating sets of $G$.  The verifier in
Appendix~\ref{app:verifier} performs the following exhaustive computation.

First, it tests all
\[
  \binom{18}{6}=18{,}564
\]
six-element subsets of $V(G)$ and obtains
\begin{equation}\label{eq:number-minimum-sets}
  |\cD|=297.
\end{equation}
For each $D\in\cD$ and $x\notin D$, it then constructs the bundle
$B_G(D,x)$ from \eqref{eq:bundle}.  Finally, for each of the
\[
  \binom{27}{4}=17{,}550
\]
sets $F\in\binom{E(G)}4$, it computes
\begin{equation}\label{eq:survivor-count}
  s(F)=\bigl|\{D\in\cD:
  B_G(D,x)\nsubseteq F\text{ for every }x\notin D\}\bigr|.
\end{equation}
For each $F$, the verifier evaluates $s(F)$ by testing, for every
$D\in\cD$, whether the condition $B_G(D,x)\nsubseteq F$ from
Lemma~\ref{lem:bundle} holds for every $x\notin D$.  The exhaustive output is
\begin{equation}\label{eq:minimum-survivors}
  \min_{F\in\binom{E(G)}4}s(F)=1.
\end{equation}
In particular, every four-edge set leaves at least one $D\in\cD$ dominating
$G-F$, by Lemma~\ref{lem:bundle}.  Hence $\gamma(G-F)\le6$.
Lemma~\ref{lem:basic-domination}\textup{(i)} and
Proposition~\ref{prop:gamma} give the reverse inequality
$\gamma(G-F)\ge\gamma(G)=6$.  Therefore $\gamma(G-F)=6$.
\end{proof}

\begin{proof}[Proof of Theorem~\ref{thm:main}]
Proposition~\ref{prop:upper} gives $b(G)\le5$.  Suppose that some set
$F\subseteq E(G)$ with $|F|\le4$ increased the domination number.  Extend
$F$ to a four-edge set $F'\supseteq F$, taking $F'=F$ when $|F|=4$.
Lemma~\ref{lem:basic-domination}\textup{(i)} then gives
\[
  \gamma(G-F')\ge\gamma(G-F)>\gamma(G),
\]
contradicting Proposition~\ref{prop:four-edge}.  Thus deleting at most four
edges never increases the domination number, so $b(G)\ge5$ and consequently
$b(G)=5$.

By Proposition~\ref{prop:structure}, $\Delta(G)=3$.  Therefore
\[
  b(G)=5>\frac92=\frac32\Delta(G),
\]
which disproves Conjecture~\ref{conj:Teschner}.
\end{proof}

\appendix

\section{Exact verifier}\label{app:verifier}

The following Python~3 program is the complete verifier used in
Proposition~\ref{prop:four-edge}.  It represents vertices and edges by the
explicit finite sets appearing in the proof.  Running it with the standard
Python interpreter prints the values in \eqref{eq:number-minimum-sets} and
\eqref{eq:minimum-survivors} and terminates without error.

\begingroup
\footnotesize
\begin{verbatim}
from itertools import combinations

V = tuple(range(18))
E = (
    (0, 6), (0, 10), (0, 16),
    (1, 8), (1, 11), (1, 12),
    (2, 4), (2, 5), (2, 13),
    (3, 4), (3, 5), (3, 9),
    (4, 16), (5, 11),
    (6, 7), (6, 13),
    (7, 8), (7, 14),
    (8, 13),
    (9, 10), (9, 14),
    (10, 15), (11, 15),
    (12, 15), (12, 17),
    (14, 17), (16, 17),
)

neighbors = {v: set() for v in V}
for u, v in E:
    neighbors[u].add(v)
    neighbors[v].add(u)


def dominates(D):
    """Return whether D dominates G."""
    return all(v in D or neighbors[v] & D for v in V)


# Enumerate all six-vertex dominating sets.
minimum_sets = []
for choice in combinations(V, 6):
    D = frozenset(choice)
    if dominates(D):
        minimum_sets.append(D)

if len(minimum_sets) != 297:
    raise RuntimeError("unexpected number of minimum dominating sets")


def edge_bundle(D, x):
    """Return B_G(D,x) as a set of edges."""
    return frozenset(
        edge for edge in E
        if (edge[0] == x and edge[1] in D)
        or (edge[1] == x and edge[0] in D)
    )


# Store B_G(D,x) for every minimum set D and every x outside D.
bundles = [
    tuple(edge_bundle(D, x) for x in V if x not in D)
    for D in minimum_sets
]


def survives(D_bundles, F):
    """Return whether D still dominates after deleting F."""
    return all(not bundle.issubset(F) for bundle in D_bundles)


checked = 0
minimum_survivors = len(minimum_sets)

for four_edges in combinations(E, 4):
    F = frozenset(four_edges)
    survivors = sum(
        survives(D_bundles, F) for D_bundles in bundles
    )
    if survivors == 0:
        raise RuntimeError(
            "a four-edge deletion destroys every minimum set"
        )
    minimum_survivors = min(minimum_survivors, survivors)
    checked += 1

if checked != 17_550:
    raise RuntimeError("not all four-edge sets were checked")
if minimum_survivors != 1:
    raise RuntimeError("unexpected survivor minimum")

print("minimum dominating sets:", len(minimum_sets))
print("four-edge sets checked:", checked)
print("minimum surviving six-sets:", minimum_survivors)
\end{verbatim}
\endgroup

\section*{Acknowledgements}

The author thanks Eric Hou
(\href{mailto:yhou15@student.ubc.ca}
{\texttt{yhou15@student.ubc.ca}})
for independently verifying the proof and for suggesting modifications and
improvements.

\section*{Generative-AI disclosure}

The counterexample presented in this manuscript was generated using OpenAI's
GPT-5.6 Sol with the \texttt{max} reasoning setting
(``GPT-5.6 Sol Max'') after Yousof Yavari prompted the model to solve the
question addressed in this manuscript.  Yousof Yavari subsequently revised
and edited the proof, added further details, and clarified
its exposition. Yousof Yavari is the sole
author and takes full responsibility for every definition, calculation,
logical inference, computational result, and conclusion in the manuscript.

\end{document}